\documentclass[reqno]{amsart}
\usepackage[english]{babel}
\usepackage[usenames,dvipsnames]{xcolor}
\usepackage{xcolor}
\usepackage[utf8]{inputenc}
\usepackage{amssymb,xspace,tensor}
\usepackage[all]{xy}
\usepackage{dsfont}
\usepackage[bookmarksnumbered,colorlinks]{hyperref}
\usepackage{graphicx}
\usepackage{enumerate}
\usepackage{mathrsfs}
\usepackage[normalem]{ulem}
\usepackage[colorinlistoftodos,italian, textwidth=4cm]{todonotes}

\newtheorem{thm}{Theorem}[section]
\newtheorem{prop}[thm]{Proposition}
\newtheorem{lem}[thm]{Lemma}
\newtheorem{cor}[thm]{Corollary}
\theoremstyle{definition}
\newtheorem{dfn}[thm]{Definition}
\theoremstyle{remark}
\newtheorem{rmk}[thm]{Remark}
\newcommand{\R}{\ensuremath{\mathbb R}\xspace}
\newcommand{\N}{\ensuremath{\mathbb N}\xspace}

\newcommand{\warp}{\ensuremath{I\tensor[_\Lambda]{\times}{}M}\xspace}

\title[Standard static  Finsler spacetimes]{Standard static  Finsler spacetimes} 
\author[E. Caponio]{Erasmo Caponio}
\address{Dipartimento di Meccanica, Matematica e Management, \hfill\break\indent
Politecnico di Bari, Via Orabona 4, 70125, Bari, Italy}
\email{caponio@poliba.it}
\author[G. Stancarone]{Giuseppe Stancarone}
\address{Dipartimento di Matematica\hfill\break\indent Universit\`a degli Studi di Bari, 
Via  Orabona 4, 70125 Bari, Italy}
\email{giuseppe.stancarone@uniba.it}



\begin{document}
\begin{abstract}
We introduce the notion of a standard static Finsler spacetime $\R\times M$ where the base  $M$ is a Finsler manifold. We prove  some results which connect causality with the Finslerian geometry  of the base extending analogous ones for static and stationary Lorentzian spacetimes.
\end{abstract}
\maketitle

\section{Introduction}
Recent  years have seen a growing interest in Finsler spacetimes. Indeed, Finslerian modifications of  general relativity have been proposed and Finsler spacetimes have been considered as possible backgrounds for the Standard Model of particles and for  quantum gravity with Lorentz-symmetry violating fields, see, e.g., \cite{PfeWol11}, \cite{PfeWol12}, \cite{vacaru12}, \cite{Vacaru12}, \cite{Bar12}, \cite{stavac}, \cite{Hohmann}, \cite{LiChang}, \cite{rajvac}, \cite{Koste}, \cite{Russell}, \cite{bipartite},   \cite{shreck}, \cite{KoStSt12}. Moreover, several authors have  started  to investigate the  geometric and  the causal structure of Finsler spacetimes, see, e.g.,  \cite{perlick06}, \cite{LPH}, \cite{PfeWol11}, \cite{torrome}, \cite{javsan14}, \cite{Minguzzi}, \cite{Minguzzi2}, \cite{amir}.  Different definitions of a Finsler spacetime metric  have been proposed in the last cited  papers but, with the exception of  \cite{LPH}, they do not include a somehow natural generalization to the Finsler realm of the notions of standard  static and standard  stationary spacetimes. This  is related to the fact that such classes of spacetime are defined on product manifolds $\tilde M=\R\times M$ where  $M$ is identifiable with a spacelike hypersurface of $\tilde M$. So a Lorentzian Finsler structure on $\tilde M$ inducing a classical Finsler structure on $M$ cannot be smoothly extended (where smooth here means  at least $C^2$) to  vectors which project trivially on $TM$, due to the lack of regularity of a Finsler function on the zero section. 

Motivated by the above problem, we define a Finsler spacetime as follows. 
Let $\tilde M$ be a $(n+1)$-dimensional smooth paracompact connected manifold, $n\geq 1$. Let us denote by $T\tilde M$ its tangent bundle and by $0$ the zero section. Let $\mathcal T \subset T\tilde M$ be a smooth real  line vector bundle on $\tilde M$ and $\mathcal T_p$ the fibre of $\mathcal T$ over $p\in \tilde M$. Let $\pi\colon T\tilde M\setminus\mathcal T\to \tilde M$ be the restriction of the canonical projection, $\tilde \pi:T\tilde M\to \tilde M$, to  $T\tilde M\setminus\mathcal T$ and let $\pi^*(T^*\tilde M)$ the pulled-back cotangent bundle over $T\tilde M\setminus \mathcal T$. Let us consider the tensor bundle $\pi^*(T^*\tilde M)\otimes\pi^*(T^*\tilde M)$  over $T\tilde M\setminus \mathcal T$ and a section $\tilde g:v\in T\tilde M\setminus \mathcal T\mapsto \tilde g_v\in T^*_{\pi(v)}\tilde M\otimes T^*_{\pi(v)}\tilde M$. We say that $\tilde  g$ is {\em symmetric} if $\tilde g_v$ is symmetric for all $v\in T\tilde M\setminus \mathcal T$. Analogously, $\tilde g$ is said {\em non-degenerate} if $\tilde g_v$ is non-degenerate  for each $v\in T\tilde M\setminus \mathcal T$ and its index will be the common index of the symmetric bilinear forms $\tilde g_v$; moreover,  $\tilde g$ will be said homogeneous if, for all $\lambda>0$ and $v\in T\tilde M\setminus \mathcal T$,  $\tilde g_{\lambda v}=\tilde g_{v}$. Finally, if the above conditions on $\tilde g$ are satisfied, we say that $\tilde g$ is the {\em vertical Hessian of a (quadratic) Finsler function} if  there exists a function $L\colon T\tilde M\to \R$ such that in natural coordinates $(x^0,x^1,\ldots, x^n,v^0,v^1,\ldots,v^n)$ of $T\tilde M$,  $\frac1 2\frac{ \partial^2L}{\partial v^i\partial v^j}(v)=\tilde g_v$, for all $v\in T\tilde M\setminus \mathcal T$.
\begin{dfn}\label{fst}
A {\em Finsler spacetime} $(\tilde M, L)$ is a smooth $(n+1)$-dimensional manifold $\tilde M$, $n\geq 1$, endowed with a smooth, symmetric, homogeneous, non-degenerate of index $1$  section $\tilde g$ of the tensor bundle $\pi^*(T^*\tilde M)\otimes \pi^*(T^*\tilde M)$ over $T\tilde M\setminus \mathcal T$, which is the vertical Hessian of a Finsler function $L$ and such that $\tilde g_w(w, w)<0$ for each $w$ in a punctured conic neighbourhood of $\mathcal T_p$ in $T_p\tilde M\setminus \{0\}$ and  for all $p\in \tilde M$. 
\end{dfn}
\begin{rmk}\label{generalized}
Clearly, the Finsler function in the above definition must be fiberwise positively homogeneous of degree $2$ and, so, it is, except for the possible lack of twice differentiability along $\mathcal T$, an  {\em indefinite  Finsler} (also called {\em Lorentz-Finsler} or simply  {\em Finsler}) function as, e.g., in  \cite{Beem70, Minguzzi, perlick06}. We could allow more generality by not prescribing the existence of such a function.
This is a quite popular  approach to  Finsler geometry: see, e.g.,   \cite{AnInMa93, MeSzTo03}, and the references therein, where   such  structures  are called {\em generalized metrics} (although in \cite{MeSzTo03} they are sections of the tensor  bundle $\tilde\pi^*(T^*\tilde M)\otimes \tilde \pi^*(T^*\tilde M)$ with base the whole $T\tilde M$). 
By homogeneity, it is easy to prove  that a generalized homogeneous metric is the vertical Hessian of a smooth Finsler function on $T\tilde M\setminus \mathcal T$ if and only if its Cartan tensor is totally symmetric, i.e. $\frac{\partial \tilde g_{ij}}{\partial v^k}(v)=\frac{\partial \tilde g_{ik}}{\partial v^j}(v)$ for all $v\in T\tilde M\setminus \mathcal T$ (see, e.g., \cite[Theorem 3.4.2.1]{AnInMa93}).  

Anyway, there are some arguments against the definition of a Finsler spacetime without a Finsler function which will be considered in Remark~\ref{against}.
\end{rmk}
\begin{rmk}\label{sign}
The requirement about the sign of $\tilde g_w(w,w)$ for $w$ in a neighbourhood of $\mathcal T$ could be weakened by allowing the existence of an open subset $\tilde M_l\subset \tilde M$ where the reverse inequality holds for any $w$ in a punctured neighbourhood of $\mathcal T_p$, for all $p\in \tilde M_l$ (and then, eventually,  the existence of a ``critical region'' where, in each punctured neighbourhood of $\mathcal T_p$, there exist vectors $w_1$ and $w_2$ such that $\tilde g_{w_1}(w_1,w_1)<0$ and $\tilde g_{w_2}(w_2,w_2)>0$). This, for example,  would be the case for a Finslerian modification of a class of spacetime called {\em $\mathrm{SSTK}$ splitting} and  studied in \cite{CJS2} (see last section).
\end{rmk}
\begin{rmk} 
Whenever $\mathcal T$ is trivial and $\tilde g$ can be smoothly extended to $\mathcal T\setminus 0$, we recover the definition of a time-orientable Finsler spacetime in \cite{Minguzzi} by choosing a no-where zero continuous section $T$ such that $T_p\in  \mathcal T$, for each $p\in \tilde M$. 
\end{rmk}
\begin{rmk}
Definition~\ref{fst} of a Finsler spacetime is a slight generalization of the ones appearing  in \cite{Beem70, perlick06,  Minguzzi} as it takes into account the problem of defining  particular Finslerian warped products. It is less general than one given in \cite{LPH} where $L$ is a.e. smooth on $T\tilde M$. We point out that all these definitions are purely kinematical in the sense that they do not take into account the problem of determining  physical reasonable Finslerian field equations (see the discussion of this problem in \cite{vacaru12}).
\end{rmk}
\begin{dfn}\label{char}
A vector $w\in T\tilde M$ is called {\em timelike} if either $w\in \mathcal T\setminus 0$ or $\tilde g_w(w,w)<0$. Moreover, $w\in T\tilde M\setminus \mathcal T$ will be said  {\em lightlike} (resp. {\em causal}, {\em spacelike}) if $\tilde g_w(w,w)=0$ (resp. $\tilde g_w(w,w)\leq 0$, either $\tilde g_w(w,w)>0$ or $w=0$). Moreover, assuming that a section $T$ as above can be chosen, we say that 
a causal  vector $w$ is {\em future-pointing }  (resp. {\em past pointing}) if either $w$ is a positive (resp. negative) multiple of $T_{\tilde\pi(w)}$ or $\tilde g_w(w,T)<0$ (resp. $\tilde g_w(w,T)>0$). A continuous piecewise smooth curve $\gamma\colon I\to \tilde M$, $I\subseteq \R$, will be called {\em timelike} (resp. {\em lightlike, causal, spacelike}) if $\dot\gamma^\pm(s)$ are  timelike (resp. lightlike, causal, spacelike) for all $s\in I$ (where, $\dot\gamma^{\pm}(s)$ denotes the right and the left derivatives of $\gamma$ at $s$). Moreover, a causal curve $\gamma\colon I\to \tilde M$ will be said {\em future-pointing },  (resp. {\em past pointing}) if $\dot\gamma^\pm(s)$ are future (resp. past) pointing for all $s\in I$. Finally  a smooth embedded hypersurface $\mathcal H\subset M$ which is transversal to $\mathcal T$ will be  called 
{\em spacelike} if for any $v\in T\mathcal H$, $\tilde g_v(v,v)>0$.
\end{dfn}
\begin{dfn}
Let $(\tilde M,L)$ be a  Finsler spacetime and  $\gamma:[a,b)\rightarrow \tilde  M $ a continuous piecewise smooth future-pointing  causal curve. Then $\gamma$ is called \emph{future extendible} if it has a continuous extension at $b$; it is \emph{future inextensible} otherwise. Analogously $\gamma:(a,b]\rightarrow \tilde M$ is called \emph{past extendible} if it can be continuously extended at $a$ and it is \emph{past inextensible} otherwise. Moreover, a future-pointing  causal curve $\gamma:(a,b)\rightarrow \tilde M$ is \emph{inextensible} if it is future and past inextensible.
\end{dfn}
Henceforth, we  will always consider continuous piecewise smooth curves, so that we will often omit to specify it when considering a curve. 
\section{Standard Static Finsler Spacetimes}\label{2}
Let us recall that a Lorentzian spacetime $(\tilde M,g)$ is said {\em static} if it is endowed with  an irrotational  timelike Killing vector field $K$. This is equivalent to say that  the orthogonal distribution to $K$ is locally integrable and then for each $p\in \tilde M$ there exists a spacelike hypersurface $S$, orthogonal to $K$, $p\in S$, and an open interval $I$ such that  the pullback of the metric $g$ by a local flow of $K$, defined in $I\times S$, is given by  $-\Lambda dt^2+g_0$, where $t\in I$, $\partial_t$ is the pullback of $K$, $\Lambda=-g(K,K)$ and $g_0$ is the Riemannian metric induced on $S$ by $g$ (see \cite[Proposition 12.38]{O'Neil}).
This local property of static spacetimes justifies the following definition: 
let $M$ be  an $n$-dimensional Riemannian manifold,  $\Lambda \colon M\to (0,+\infty)$ a smooth, positive function on $M$ and $I\subseteq \mathbb R$ an open interval. The  warped product \warp, i.e. the manifold 
$I\times M$ endowed with the Lorentzian metric $g=-\Lambda dt^2+g_0$, where $g_0$ is the pullback on $I\times M$ of the Riemannian metric on $M$,   is a spacetime called {\em standard static} (see \cite[Definition 12.36]{O'Neil}).

The conformal Riemannian metric on $M$, $g_0/\Lambda$ is called {\em optical metric}. It plays a fundamental role in the study of light rays of \warp because its  pregeodesics are the projections on $M$ of the light rays in $(I\times M, g)$ (see \cite{AbCaLa88}, \cite{GibWar09},  \cite{Perlic04}). Moreover, many of the causal properties of the spacetime \warp are encoded in the geometry of the conformal manifold $(M, g_0/\Lambda)$. For example, global hyperbolicity of \warp and the fact that the the slices $\{t_0\}\times M$, $t_0\in I$, are Cauchy hypersurfaces are both equivalent to the completeness  of the optical metric (these properties follows by the  conformal invariance of causal properties plus  Theorem 3.67 and  Theorem 3.69-(1) with $f\equiv 1$ of \cite{Beem}).  

Let us extend the above picture to Finsler spacetimes. First we need the notion of a Killing vector field.
\begin{dfn}
	Let $(\tilde M, L)$ be a Finsler spacetime  and $K$ be a vector field on $\tilde M$. Let $\psi$ be the flow of $K$. We say that $K$ is a {\em Killing vector field} if for each $v\in T\tilde M\setminus \mathcal T$ and for all $v_1,v_2\in  T_{\pi(v)}\tilde M$, we have: 
	\begin{equation}
	\tilde g_{d\psi_{\bar t}(v)}(d \psi_{\bar t}(v_1),d \psi_{\bar t}(v_2))= \tilde g_{v}(v_1,v_2),
	\end{equation}
	for any $\bar t\in\R$ such that the stage $\psi_{\bar t}$ is well defined in a neighbourhood $U\subset \tilde M$ of $\pi(v)$.
\end{dfn}
\begin{rmk}
As proved in \cite[Proposition 5.2]{Lovas04}, this is equivalent to the fact that $K$ is a Killing vector field for $\tilde g$, in the sense that the Lie derivative $\mathcal L_{K}\tilde g=0$ (see \cite[p.136]{Lovas04} for the definition of the Lie derivative  over  the tensor bundle  $\pi^*(T^*\tilde M)\otimes \pi^*(T^*\tilde M)$; actually in \cite{Lovas04} the base of the tensor bundle is the slit tangent bundle $T\tilde M\setminus 0$ but the reader can check the validity of the equivalence when considering $T\tilde M\setminus \mathcal T$).
\end{rmk}

\begin{dfn}
	We say that a Finsler spacetime $(\tilde M, L)$ is {\em static} if there exists a timelike  Killing vector field $K$  such that the distribution of hyperplanes   is integrable, where $\partial_vL(K)$ denotes the one-form on $\tilde M$ given by $\frac{\partial L}{\partial v^i}(K)dx^i$.
\end{dfn}
\begin{rmk}
	Observe that  the above definition is well posed since  $L$ is at least a  $C^1$ function on $T\tilde M$.  In particular, it works also for a  smooth global section $T$ of $\mathcal T$ (in the case when  $\mathcal T$ is trivial) or, more generally, for a vector field $K$ such that $K_p\in \mathcal T_p$ for some $p\in \tilde M$. 
\end{rmk}
\begin{dfn}\label{staticFinsler}
	We say that a Finsler spacetime is {\em standard static} if there exist  a smooth non vanishing global section $T$ of $\mathcal T$, a Finsler manifold $(M,F)$, a positive function $\Lambda$ on $M$  and a  smooth diffeomorphism $f\colon \R\times M\to \tilde M$, $f=f(t,x)$, such that  $\partial_t=f^*(T)$ and $L(f_*(\tau, v))=-\Lambda \tau^2+ F^2(v)$.
\end{dfn}
\begin{rmk}
	The definition of a standard static Finsler spacetime (although  called there  static Finsler spacetime) appeared first in \cite[Definition 2]{LPH}. In \cite{LiChang} the solution of the vacuum (Finslerian) field equations, introduced in the same paper, is standard static in the region where a certain coefficient $B$ is positive provided that a constant $a$ is also positive (see \cite[Eqs. (16)-(17)]{LiChang}).   
\end{rmk}
\begin{rmk}
	Another static  future-pointing Killing vector field $K$ will be said {\em standard} if the above conditions hold relatively to $K$, i.e there exist a manifold $(M', G)$ and  a  diffeomorphism $f':\R\times M'\to \tilde M$, $f'=f'(t', x')$, such that $\partial_{t'}=(f')^*(K)$ and $L(f'_*(\tau', v'))=-\Lambda^K \tau'^2+G^2(v')$. 
\end{rmk}
The existence of a standard static vector field is a very rigid condition in comparison to  the Lorentzian case
(\cite{sansen,alroru}) where some topological assumptions on  the base $M$ are needed in order to get uniqueness. Indeed we have the following:
\begin{prop}
	If a static Finsler spacetime admits a standard splitting (i.e the static vector field is standard) then it is unique up to rescaling $t\mapsto t/a$, $T\mapsto aT$, $a\in (0,+\infty)$,  of the coordinate $t$ and  of the vector field $T$ and up to Finslerian isometries of $(M,F)$.
\end{prop}
\begin{proof}
	Assume that another static standard  Killing vector field $K$ exists and let  $f'\colon \R\times M'\to \tilde M$ be a smooth diffeomorphism such that  $(f')^*(K)=\partial_{t'}$ and $L(f'_*(\tau', v'))=-\Lambda^K \tau'^2+G^2(v')$, so that $G^2(v')=L(f'_*(0,v'))$ for all $v'\in TM'$. Let $L':= L\circ f'_*$. Then $L=L'\circ (f')^*$,  hence $L$ is not twice differentiable along the line bundle $\mathcal K$ defined by $K$, because $L'$ is not so along the one defined by $\partial_{t'}$. This is possible if and only if $\mathcal K=\mathcal T$, which is equivalent to the fact that $K$ is collinear to $T$ at every point in $\tilde M$. Since both $K$ and $T$ are Killing vector fields for $\tilde g$ and they belong to the same timelike cone necessarily they are proportional, i.e. there exists a positive constant  $a$ such that $K=aT$. This follows as in the semi-Riemannian   case,  by using the fact that $\mathcal L_{T}\tilde g= \mathcal L_{K}\tilde g=0$ (see \cite[p. 136,  Definition 5.1 and Proposition 5.2]{Lovas04}). Thus the diffeomorphism $(f')^{-1}\circ f$ has first component equal to $t\mapsto t/a$ while the second one induces  a diffeomorphism $\phi$ between  $M$ and $M'$ such that  $G^2(\phi_*(v))=F^2(v)$ for all $v\in TM$.\end{proof}
\begin{rmk}
Henceforth, we will identify a standard static Finsler spacetime $(\tilde M,L)$ with the product manifold $\R\times M$ endowed with the Finsler function $L(\tau,v)=-\Lambda \tau^2+F^2(v)$, where $\Lambda$ and $F$ are, respectively, a positive function and a Finsler metric on $M$. 	
\end{rmk}

\begin{rmk}\label{orto}
 Observe  that $\partial_t$ is $\tilde g_{(\tau,v)}$-orthogonal to $\{0\}\times T_xM$ for any $(\tau, v)\in T\tilde M\setminus \mathcal T$, $x=\pi^M(v)$, where $\pi^M:TM\to M$ is the canonical projection. These facts justify, by analogy with the Lorentzian case, the name ``standard static'' given to the class of Finsler spacetime in Definition~\ref{staticFinsler}. Anyway, as observed in \cite[Example 1, Remark 3]{Minguzzi}, differently from the Lorentzian case, a Finsler spacetime can be static without being  locally standard static either (in the example of  \cite{Minguzzi}, the Killing vector field is not $\tilde g_{\tilde v}$-orthogonal to $\ker(\partial_v L(K_{\pi(\tilde v)}))$, for all $\tilde v\in \ker(\partial_v L(K))$). 
\end{rmk}
\begin{rmk}\label{against}
Clearly, in  Definition~\ref{staticFinsler}, we could allow more generality by taking  a positive definite, homogeneous, generalized metric $g$ on $M$ (recall Remark~\ref{generalized}) and/or a function $\tilde \Lambda\colon TM\setminus 0\to (0,+\infty)$ which is fiberwise positively homogeneous of degree $0$, i.e $\tilde\Lambda(\lambda v)= \tilde \Lambda (v)$ for each $v\in TM\setminus 0$ and $\lambda>0$. Let us focus on the latter case.  Observe, first, that the generalized metric $\tilde g$ will not come, in general,  from a Finsler function. 
In such a generalized  standard static Finsler spacetime the set of future-pointing causal vectors $J$ at a point $(t,x)$ is given by the non-zero vectors $(\tau, v)$ satisfying $\tau \geq F(v)/\sqrt{\Lambda(v)}$. Being $\Lambda$ positively homogeneous of degree $0$ and positive it satisfies $C_1(x)\leq \Lambda(v)\leq C_2(x)$, for some positive constants $C_1(x),C_2(x)$, and for all $v\in T_xM\setminus \{0\}$, so that $F/\sqrt{\Lambda}$ can be extended by continuity in $0$. Thus, $J$ is connected. Nevertheless, it is, in general, non-convex  (see, e.g.,  Figure~\ref{nonconv}).
This is in contrast to what happens in Finsler spacetimes defined through a Finsler function where the connected components of $J$ are convex (see \cite[Theorem 2]{Minguzzi}) and
it  should be considered as a serious argument  against the definition of a Finsler spacetime through a generalized metric which does not come from a quadratic Finsler function. In fact, in this case, a reverse Cauchy-Schwarz inequality (see Proposition~\ref{rCS} below) cannot hold and there exist causal vectors $(\tau_1,v), (\tau_2,w)$ which are in the same connected component of the set of causal vectors and  such that $\tilde g_{(\tau_1,v_1)}((\tau_1,v_1),(\tau_2,w))>0$ (see Figure~\ref{nonconv}).
\begin{figure}[h]
\vspace{-0.5cm}
\includegraphics[scale=0.4]{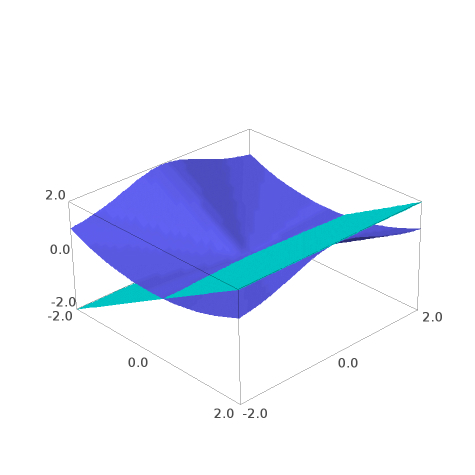}
\caption{The set of the f. p. lightlike  vectors (in blue) in $\R\times \R^2$ with the (flat)  static metric $\tilde g=-e^{\frac{4 \, v_2^{2}}{v_1^{2} + v_2^{2}}}dt^2 +dx^2+dy^2$. In cyan, it is represented the plane of vectors $(\tau,v_1,v_2)$ which are $\tilde g_{(1,1,0)}$-orthogonal to the lightlike vector $(1,1,0)$.}\label{nonconv}
\end{figure}
\end{rmk}
Let us determine  the geodesic equations in a standard static Finsler spacetime. Given a manifold $N$ and two points $p,q\in N$, let $\Omega_{pq}(N)$ be the set of the continuous  piecewise smooth curve $\gamma$ on $N$ parametrized on a given  interval $[a,b]\subset \R$ and connecting $p$ to $q$ (i.e. $\gamma(a)=p$, $\gamma(b)=q$). If $\gamma\in\Omega_{pq}(N)$, we call a {\em (proper) variation} of $\gamma$ a continuous two-parameter map $\psi\colon (\varepsilon,\varepsilon)\times [a,b]\to N$ such that $\psi(0,s)=\gamma(s)$, for all $s\in [a,b]$, $\psi(w,\cdot)\in \Omega_{pq}(N)$ and  there exists a subdivision $a=s_0<s_1<\ldots, s_k=b$ of the interval $[a,b]$ for which  $\psi|_{(-\varepsilon,\varepsilon)\times[s_{j-1},s_j]}$ is smooth for all $j\in \{1,\ldots,k\}$.
Clearly, we can define classes of proper variations of $\gamma$ as those sharing the same  {\em variational vector field} $Z$. This is, by definition,  a continuous piecewise smooth vector field along $\gamma$ such that $Z(a)=0=Z(b)$ and $Z(s)=\frac{\partial \psi}{\partial r}(0,s)$. By considering any auxiliary Riemannian metric $h$ on $N$, we see that   each variational vector field $Z$ along $\gamma$ individuates a variation (and then also a class of them) by setting $\psi(w,s):=\exp_{\gamma(s)}(wZ(s))$, for $|w|<\varepsilon$ small enough.

Let us consider the energy functional 
\[
E\colon \Omega_{pq}(\tilde M)\to \R,\quad\quad E(\gamma)=\frac{1}{2}\int_a^b\big(-\Lambda(\sigma)\dot\zeta^2+F^2(\dot\sigma)\big)ds.
\]
As $\tilde M$ splits as $\R\times M$, the path space $\Omega_{pq}(\tilde M)$ is identifiable with the product $\Omega_{t_pt_q}(\R)\times \Omega_{x_px_q}(M)$, where $(t_p,x_p)=p$ and $(t_q,x_q)=q$ and any curve $\gamma\in  \Omega_{pq}(\tilde M)$ has two components $\gamma(s)=(\theta(s),\sigma(s))$. 
\begin{dfn}
A continuous piecewise smooth curve $\gamma\colon[a,b]\to\tilde M$ is a (affinely parametrized) {\em geodesic} of $(\tilde M, L)$ if  it is a critical point of the energy functional, i.e. if  $\frac{d}{dr}(E(\psi(r,\cdot))|_{r=0}=0$, for all proper variations $\psi$ of $\gamma$.  
\end{dfn}
\begin{thm}
A curve $\gamma\colon[a,b]\to\tilde M$, $\gamma(s)=(\theta(s),\sigma(s))$, is a geodesic of $(\tilde M,L)$ if and only if 
the following equations are satisfied in local natural coordinates $(t,x^1,\ldots, x_n, \tau, v^1,\ldots,v^n)$ on $T\tilde M$:
\begin{align}
-&\frac{\partial \Lambda}{\partial x^i}(\sigma)\dot\theta^2+\frac{\partial F^2}{\partial x^i}(\dot\sigma)-\frac{d}{ds}\left(\frac{\partial F^2}{\partial v^i}(\dot\sigma)\right)=0,\quad\quad i=1,\ldots,n\label{Meq}\\
&\Lambda(\sigma)\dot\theta=\mathrm{const.}\label{teq}
\end{align}
$\sigma$ and $\theta$ are, respectively,  $C^1$ and $C^2$ on $[a,b]$ and   there exists a constant $C\in \R$ such that $L(\dot\gamma(s))=C$, for all $s\in [a,b]$. 

Moreover,   if $\gamma$ is  a non-constant geodesic then: $(a)$ if it is spacelike or lightlike (i.e.  $C\geq 0$) then $\dot\sigma$ never vanishes and $\gamma$ is smooth; $(b)$ if $\sigma$ is  constant equal to $x_0\in M$ on the whole interval $[a,b]$ then  $C<0$, $d\Lambda(x_0)=0$ and $\dot \theta$ is constant too; vice versa, if  $d\Lambda(x_0)=0$ then,  for each $\theta_0\in \R$ and $m\neq 0$, the curve $s\in[a,b]\mapsto (\theta_0+m(s-a), x_0)\in \tilde M$ is a  timelike geodesic.
\end{thm}
\begin{rmk}
In particular  the function $s\in[a,b]\mapsto \frac{\partial F^2}{\partial v^i}(\dot\sigma(s))$ is $C^1$ on $[a,b]$, in fact  it is differentiable also at the instants $s$ where $\dot\sigma(s)=0$.
\end{rmk}
\begin{proof}
By considering variational vector fields $Z$ which are, respectively, of the type $(Y,0)$ and $(0, W)$ and have compact support in any open maximal interval $I$ where $\sigma$ is smooth and $\dot\sigma(s)\neq 0$,  we deduce  by standard arguments that in local coordinates $(t,x^1,\ldots, x^n,\tau, v^1, \ldots, v^n)$ of $T\tilde M$, any critical point $\gamma=(\theta,\sigma)$ of $E$ satisfies, in such interval $I$, equations \eqref{Meq} and \eqref{teq}.
Actually, it can be easily seen that \eqref{teq} holds on all $[a,b]$ and then,  as $\sigma$ is continuous, $\theta$ is $C^1$ on $[a,b]$.   
If $\bar s\in [a,b]$ is an isolated zero of $\dot \sigma(s)$, and  $\sigma$ is smooth in a neighbourhood $J$ of $\bar s$ then, as the functions $-\frac{1}{2}\frac{\partial \Lambda}{\partial x^i}(\sigma)\dot\theta^2+\frac{\partial F^2}{\partial x^i}(\dot\sigma)$ and $\frac{\partial F^2}{\partial v^i}(\dot\sigma(s))$ are continuous in $J$, we deduce, from the fact that \eqref{Meq} is satisfied in a left and in a right neighbourhood of $\bar s$, that $s\mapsto\frac{\partial F^2}{\partial v^i}(\dot\sigma(s))$ is differentiable with continuous derivative at $\bar s$  and \eqref{Meq} is satisfied also at $\bar s$. By this information  at isolated zeroes of $\dot\sigma$, we see that the same argument applies at a zero where $\sigma$ is smooth and which is an accumulation point of isolated zeroes.  

On the other hand, at the instants $s_j$, $j\in \{0,\ldots,k\}$, where $\dot\sigma$ has a break,  by taking any vector $w_j \in T_{\sigma(s_j)}M$, $j\in\{1,\ldots,k-1\}$, and a  variational vector field $(0,W_j)$, such that $W_j(s_j)=w_j$ and $W\equiv 0$ outside a small neighbourhood of $s_j$, we get, using  that \eqref{Meq} is satisfied both in $[s_{j-1}, s_j]$ and $[s_{j}, s_{j+1}]$, 
\[\frac{\partial F^2}{\partial v^i}(\dot\sigma^-(s_j))w_j^i=\frac{\partial F^2}{\partial v^i}(\dot\sigma^+(s_j))w_j^i,\]
hence $\frac{\partial F^2}{\partial v}(\dot\sigma^-(s_j))=\frac{\partial F^2}{\partial v}(\dot\sigma^+(s_j))$. Being the map $v\in T_x M\mapsto \frac{\partial F^2}{\partial v}(v)\in T_x^*M$ a homeomorphism  (see e.g. \cite[p.373]{CaJaMa11}) we deduce that $\dot\sigma^-(s_j)=\dot\sigma^+(s_j)$. Thus, $\sigma$ is a $C^1$ curve on $[a,b]$ and, from \eqref{teq}, $\theta$ is a $C^2$ function.  Then, from $s$-independence of the Lagrangian $L$ and the fact that it is positively homogeneous of degree $2$, we know that   $L(\dot\gamma(s))=\frac{\partial L}{\partial v}(\dot\gamma(s))[\dot\gamma(s)] -L(\dot\gamma(s))=\mathrm{const.}:=C_j\in\R$ in each interval $[s_j,s_{j+1}]$ where $\gamma$ is smooth.  As $\gamma$ is a $C^1$ curve on $[a,b]$, the constants $C_j$ must agree, i.e. $L(\dot\gamma(s))=\mathrm{const.}:=C\in\R$ on $[a,b]$. 

The converse clearly follows by observing that each $C^1$ curve $\gamma$ which solves \eqref{Meq} and \eqref{teq} must necessarily be a critical point of the energy functional.

Last part of the theorem follows by observing that, being $\Lambda>0$, if $C>0$ then $\dot\sigma\neq 0$ everywhere in $[a,b]$ and if  $\gamma$ is  non-constant and $C=0$
then at each instant $\bar s\in [a,b]$ where $\dot\sigma(\bar s)=0$, necessarily, also $\dot\theta(\bar s)=0$ and, from \eqref{teq}, $\dot\theta\equiv 0$ and also $F^2(\dot\sigma)\equiv 0$, i.e. $\dot\sigma\equiv 0$, a contradiction 
since $\gamma$ was not constant. As  the vertical Hessian of $F^2$ is invertible at each non-zero vector, from \eqref{Meq} written in normal form, we deduce  that if $C\geq 0$ then  $\sigma$ is smooth (i.e. it is at least twice differentiable also at the instants $s_j$, $j\in\{1,\ldots,k\}$. 
Finally, if $\sigma$ is constant and equal to $x_0\in M$ then necessarily $C<0$ and, from \eqref{teq},   $\dot \theta$ is a non-zero constant while, from \eqref{Meq}, $\frac{\partial \Lambda}{\partial x^i}(x_0)=0$. Analogously, one can check that last statement holds true.  
\end{proof}
\begin{dfn}
A \em {pregeodesic} of $(\tilde M,L)$ is any $C^1$ curve $\gamma\colon [c,d]\to \tilde M$ admitting  a  reparametrization $\varphi:[a,b]\to [c,d]$ which is $C^1$, regular and orientation preserving (i.e. $\dot\varphi>0$) such that $\gamma\circ\varphi$ is a geodesic. 
\end{dfn}
As lightlike vectors $w=(v,\tau)$ of $(\R\times M, \tilde g)$ are defined by the equation
$g_{v}(v,v)-\Lambda \tau^2=0$, we give, in analogy to the Lorentzian case, the following definition:
\begin{dfn}\label{optical}
The {\em optical ``metric''} of  a standard static Finsler spacetime is the positive definite homogeneous section $v\in TM\setminus 0\mapsto \frac{1}{\Lambda}g_v$. 
\end{dfn}
\begin{rmk}
Notice that the optical metric  is the fundamental tensor of the  Finsler metric $\tilde F=F/\sqrt{\Lambda}$ on $M$. 
\end{rmk}
Definition~\ref{optical} is justified by the following result:
\begin{prop}\label{fermat}
Let $(\tilde M, L)$ be a standard static Finsler spacetime. A curve $\gamma\colon [a,b]\to \tilde M$, $\gamma(s)=(\theta(s),\sigma(s))$ is a future-pointing lightlike geodesic if and only if $\sigma$ is a (non-constant) pregeodesic of the Finsler metric $F/\sqrt{\Lambda}$ on $M$ parametrized with $\sqrt{\Lambda(\sigma)}F(\dot\sigma)=\mathrm{const.}$ and $\theta(s)=\theta(a)+\int_a^s\frac{F(\dot\sigma)}{\sqrt{\Lambda(\sigma)}}d \tau$.
\end{prop}
\begin{proof}
Let us assume that $\gamma$ is a future-pointing lightlike geodesic of $(\tilde M, L)$. By \eqref{teq} and the fact that $\gamma$ is lightlike we get  that $\sqrt{\Lambda(\sigma)}F(\dot\sigma)$ must be  constant on $[a,b]$. Hence, Eq.~\eqref{Meq} is equivalent to 
\begin{multline*}
2\sqrt{\Lambda(\sigma)}F(\dot\sigma)\left(-\frac{1}{2}\frac{\partial \Lambda}{\partial x^i}(\sigma) \frac{F(\dot\sigma)}{(\Lambda(\sigma))^{3/2}}+
\frac{1}{\sqrt{\Lambda(\sigma)}}\frac{\partial F}{\partial x^i}(\dot\sigma)\right.\\ \left.-\frac{d}{ds}\left(\frac{1}{\sqrt{\Lambda(\sigma)}}\frac{\partial F}{\partial v^i}(\dot\sigma)\right)\right)=0.
\end{multline*}
As $\sqrt{\Lambda(\sigma)}F(\dot\sigma)$ is positive, the above equation is satisfied if and only if 
\begin{equation}\label{geoFconf}
-\frac{1}{2}\frac{\partial \Lambda}{\partial x^i}(\sigma) \frac{F(\dot\sigma)}{(\Lambda(\sigma))^{3/2}}+
\frac{1}{\sqrt{\Lambda(\sigma)}}\frac{\partial F}{\partial x^i}(\dot\sigma)-\frac{d}{ds}\left(\frac{1}{\sqrt{\Lambda(\sigma)}}\frac{\partial F}{\partial v^i}(\dot\sigma)\right)=0.
\end{equation}
This is the Euler-Lagrange equation (in natural local coordinates of $TM$) of the length functional  associated to $F/\sqrt{\Lambda}$, hence $\sigma$ is a pregeodesic of the Finsler manifold $(M, F/\sqrt{\Lambda})$. The converse immediately follows by using invariance under  $C^1$, regular, orientation preserving reparametrizations of the solutions of \eqref{geoFconf}.
\end{proof}
\begin{rmk}\label{arrival}
Notice that the length functional of the Finsler  metric $F/\sqrt{\Lambda}$ on the path space $\Omega_{x_px_q}(M)$ coincides,  up to a constant,  with the {\em arrival time functional} $T_{pl_{x_q}}$ of the standard static Finsler spacetime
$(\tilde M, L)$; this is  the functional  defined on the set of the future-pointing   lightlike curves $\gamma$ connecting $p$ with the line $l_{x_{q}}=s\mapsto (s, x_q)$ and defined as $T_{pl_{x_q}}(\gamma)=t(\gamma(b))$.     Hence, Proposition~\ref{fermat} can be interpreted   as a Fermat's principle for light rays in the Finsler spacetime $(\tilde M, L)$, namely {\em the critical point of $T_{pl_{x_q}}$ are all and only the  lightlike pregeodesic of $(\tilde M, L)$}. For a general version of the Fermat's principle in a Finsler spacetime defined through a  quadratic Finsler function which is smooth on $T\tilde M\setminus 0$ see \cite[Section 4]{perlick06}.
\end{rmk}
\begin{rmk}
Let  $\tilde L\colon T\tilde  M\to \mathbb R$ be the  function given by $\tilde L(\tau,v)=\frac{1}{\Lambda}F^2(v)-\tau^2$. We call the pair $(\tilde M,\tilde L)$ the \emph{ultrastatic Finsler spacetime} associated to 
$(\tilde M, L)$ and we  denote by  $G$ the  square of $\tilde F$: $G=\frac{1}{\Lambda}F^2$ and by $\tilde d$ the distance associated to $\tilde F$.
From Proposition~\ref{fermat} we immediately deduce that the Finsler spacetime $(\tilde M, L)$ and the ultrastatic one $(\tilde M,\tilde L)$ share  the same lightlike  pregeodesics. In other words, invariance of lightlike geodesics under conformal changes of the metric (which is a  fundamental property of Lorentzian spacetime) also holds  in  the class of standard static Finsler spacetime under conformal factors depending only on $x\in M$. \end{rmk}

It has been known at least since  \cite{Beem70} that the causal structure of a tangent space in a Finsler spacetime can be weird. In fact, \cite{Beem70} contains  some two-dimensional examples of Finsler spacetime where the causal cone at a point $x$ (i.e. the the set of the vectors $v\in T_x \tilde M\setminus \{0\}$ which are causal)  has more than two connected component. Recently, in  \cite{Minguzzi}, it has been proved that such pathologies are confined in dimension two if the Finsler spacetime is time oriented and the Finsler spacetime metric is reversible. For a standard static spacetime $(\tilde M,L)$ things are simpler as the following proposition shows:
\begin{lem}\label{convexcausal}
Let $(\tilde M,L)$ be a standard static Finsler spacetime. Then the set of future-pointing causal vectors at a point $(t,x)\in \tilde M$ has only one connected  convex component. Moreover, for each $c>0$, the set $J(c)$ of the future-pointing timelike vectors in $T_{(t,x)}\tilde M$ such that  $L(\tilde v)\leq -c\}$ is also connected and strictly convex.
\end{lem}
\begin{proof}
Observe that, by definition, future-pointing causal vectors $\tilde v=(\tau, v)\in T_{(t,x)}\tilde M$ are all and only the non-zero vectors satisfying
$\tau\geq \tilde F(v)$ (recall that, for all $\tilde v\in T \tilde M\setminus \mathcal T$,  as $\tilde{g}_{\tilde{v}}(\tilde{v},\frac{\partial}{\partial t})\leq 0$, $\tau$ has non-negative sign). Being $\tilde F$ continuous on $T_xM$ and fiberwise convex (see, e.g \cite[Lemma 1.2.2]{Shen01}) its epigraph in $T_x M$ is connected and convex. For the last part of the proposition, observe that $\tilde v\in J(c)$ if and only if $\tau\geq \sqrt{G(v)+\alpha}$ where $\alpha=c/\Lambda(x)$.  
A simple computation shows that for each $v\in T_xM\setminus \{0\}$, the fiberwise Hessian of   $\sqrt{G+\alpha}$ at $v$ is given by
\[\frac{\frac{\partial \tilde F}{\partial v^i}(v)\frac{\partial \tilde F}{\partial v^j}(v)}{\sqrt{G(v)+\alpha}}\left(1-\frac{G(v)}{G(v)+\alpha}\right)+\frac{\tilde F(v)}{\sqrt{G(v)+\alpha}}\frac{\partial^2\tilde F}{\partial v^i\partial v^j}(v).\]
Since $\frac{\partial^2\tilde F}{\partial v^i\partial v^j}(v)w^iw^j\geq 0$ and it is equal to $0$ if and only if  $w=\lambda v$, for some $\lambda\in \R$ and, moreover,  $\frac{\partial \tilde F}{\partial v^i}(v)v^i=F(v)>0$, we get that the Hessian of $\sqrt{G+\alpha}$ is positive definite for any $v\in T_xM\setminus 0$. At $v=0$, observe that the differential of $\sqrt{G+\alpha}$ is $0$ and  $\sqrt{G(v)+\alpha}>\sqrt{\alpha}=\sqrt{G(0)+\alpha}$. Thus, being $\sqrt{G+\alpha}$ a $C^1$ function on $T_x M$ we conclude that it is strictly convex and its epigraph is also (connected) and strictly convex. 
\end{proof}
By Lemma~\ref{convexcausal} we get the following reverse Cauchy-Schwarz inequality:  
\begin{prop}\label{rCS}
Let $(\tilde M, L)$ be a standard static spacetime and $\tilde{v}, \tilde w \in T_{(t,x)}\tilde M$,  future-pointing causal vectors.
Then
\begin{equation}\label{Cauchy}
-\frac 1 2\frac{\partial L}{\partial \tilde v^i}(\tilde v)\tilde w^i\ge\sqrt{-L(\tilde{v})}\sqrt{-L(\tilde{w})},
\end{equation}
with equality if and only if $\tilde v$ and $\tilde w$ are proportional.
\end{prop}
\begin{proof}
Recalling that $L$ is $C^1$ in $T\tilde M$ and smooth outside $\mathcal T$, the same proof of \cite[Theorem 3]{Minguzzi} gives 
\[-\tilde g_{\tilde v}(\tilde v,\tilde w)\ge\sqrt{-L(\tilde{v})}\sqrt{-L(\tilde{w})},\]
for any $\tilde v\in T_{(t,x)}\tilde M\setminus \mathcal T_{(t,x)}$ and $\tilde w\in T_{(t,x)}\tilde M$, with equality if only if $\tilde v$ and $\tilde w$ are proportional. Then \eqref{Cauchy} extends to any $\tilde v\in T_{(t,x)}\tilde M$ by continuity, recalling that, by homogeneity,
$-\frac1 2 \frac{\partial L}{\partial \tilde v^i}(\tilde v)\tilde w^i=-\tilde g_{\tilde v}(\tilde v,\tilde w)$. Moreover if $\tilde v\in \mathcal T$, so $\tilde v=(\tau_1,0)$ for some $\tau_1>0$ then \eqref{Cauchy} can hold with equality if and only if $\tilde w\in \mathcal T$. In fact if  $\tilde w=(\tau_2, w)$, with $w\neq 0$, then  
the right-hand side of \eqref{Cauchy} is  equal to $\tau_1\sqrt{\Lambda(x)}\sqrt{\Lambda(x)\tau_2^2-F^2(w)}$ which is strictly less than the left-hand side equal to $\Lambda(x)\tau_1\tau_2$.
\end{proof}
\begin{rmk}
In particular, the above proposition implies that if  $\tilde{v},\tilde{w}\in T_{(t,x)}\tilde M\setminus\mathcal T_{(t,x)}$,  are future-pointing causal vectors   then  $\tilde{g}_{\tilde{w}}(\tilde{w},\tilde{v})\leq 0$ and $\tilde{g}_{\tilde{v}}(\tilde{v},\tilde{w})\leq 0$; moreover $\tilde g_{\tilde v}(\tilde v, \tilde w)=0$ if and only $\tilde v$ and $\tilde w$ are proportional and lightlike.
\end{rmk}
\section{Causality}
In a Lorentzian spacetime  $(M,g)$,  two events $p, q\in M$ are said  \emph{chronologically} (resp. {\em causally}) {\em related} and denoted with $p\ll q$ (resp. $p\leq q$), if there exists a  future-pointing timelike (resp. causal) curve $\gamma$ from $p$ to $q$; $p$ is said {\em strictly causally related} to $q$, denoted with $p<q$,  if $p\leq q $ and $p\neq q$. The \emph{chronological} (resp. \emph{causal}) {\em future} of $p\in M$ is defined as $I^{+}(p):=\{q\in M: p\ll q\}$ ($J^{+}(p)=\{q\in M: p\leq q\}$). Analogous notions appear reversing the binary relation $\ll$, namely we have the \emph{chronological } (resp. \emph{causal}) {\em past} $I^{-}(p)=\{q\in M: q\ll p\}$ ($J^{-}(p)=\{q\in M: q\leq p\}$). For further details see \cite{Beem}, \cite{O'Neil}. 

The above definitions can be trivially extended to a Finsler spacetime $(\tilde M, \tilde g)$ but  notice that in the  Lorentzian setting the chronological and the causal past of an event $p$ can be equivalently defined by considering past-pointing timelike and causal curves starting at $p$. Clearly, in a Finsler spacetime, this is  true  only when the generalized metric is  absolutely homogeneous, i.e. $\tilde g_{v}=\tilde g_{-v}$ for any $v\in T\tilde M\setminus \mathcal T$. In the general case,  given a  future-pointing causal vector $v$ or a future-pointing causal curve,  $-v$ and the curve parametrized with the opposite orientation can be spacelike. So the chronological (resp. causal) past of an event $p$ will be defined by considering only future-pointing timelike (resp. causal) curves arriving at $p$.

One immediate consequence of the definitions of the chronological and the causal future of a point is that, for all $p\in \tilde M$:
\begin{equation}\label{inclusion}
I^+(p)\subset J^+(p), \quad \quad  I^-(p)\subset J^-(p).
\end{equation}
\begin{rmk}
After \cite{CaJaMa11} and mostly  \cite{CJS,CapGerSan12}, it is clear that Finsler geometry  plays a prominent role in the description of the causal structure of a class of stationary spacetime  which generalizes the standard static one in the following sense: under the same notation as in Section~\ref{2}, let us also consider  a one-form on $M$; then define a Lorentzian metric $g$ on the product  $\R\times M$ as $g=g_0+\omega\otimes dt+dt\otimes \omega-\Lambda dt^2$, where $\omega$ is the pullback on $\R\times M$ of the one form on $M$.   A Finsler metric $R$ of Randers type emerges then as the optical metric of $(\R\times M, g)$, $R=\frac{1}{\Lambda}(\omega+\sqrt{\Lambda g_0+\omega\otimes\omega})$. Analogously, 
the metric structure  of the Finsler manifold $(M,F/\sqrt{\Lambda})$ associated to a static Finsler spacetime $(\tilde M,L)$ can be related to its  causal structure. As a first example, the following proposition analogous to \cite[Prop. 4.2]{CJS} holds:
\end{rmk}
\begin{prop}\label{forwball}
Let $(\tilde M,L)$ be a standard  static Finsler spacetime. For all $p_0=(t_0,x_0)\in \tilde M$ we have:
\begin{align}
&I^{+}(p_0)=\bigcup_{r>0}\left(\{t_0+r\}\times B^{+}(x_0,r)\right), \label{ipiu}\\
&I^{-}(p_0)=\bigcup_{r>0}\left(\{t_0-r\}\times B^{-}(x_0,r)\right),\nonumber
\end{align}
where $B^{+}(x_0,r)$ and $B^-(x_0,r)$ denote, respectively, the forward and the backward open ball (see \cite[\S 6.2 B]{Chern}) of centre $x_0$ and radius $r$ of the Finsler metric $\tilde F=F/\sqrt{\Lambda}$.   Moreover, $I^{\pm}(p_0)$ are open subsets of $\tilde M$.
\end{prop}
\begin{proof}
Let us reasoning  only for $I^+$ as the statements for $I^-$ can be proved analogously. Let $(t,x)\in I^{+}(t_0,x_0)$ and $\gamma(s)=(\theta(s),\sigma(s))$ be a timelike curve joining $(t_0,x_0)$  to $(t,x)$, so we have that $\theta $ is an increasing function and $\tilde F(\dot\sigma)<\dot\theta$. Integrating this inequality, we get $\tilde d(x_0,x)<t-t_0$. Hence, $x\in B^{+}(x_0,t-t_0)$ and we conclude $(t,x)\in\{t_0+r\}\times B^{+}(x_0,r)$, $r=t-t_0$. Conversely, let $x\in B^{+}(x_0,r)$, for some $r>0$, and $\sigma:[0,a]\rightarrow M$  be a unit ($F/\sqrt{\Lambda}$)-speed   curve joining $x_0$ to $x$, such that $a<r$. The curve  $\gamma(s)=(t_0+\frac{r}{a}s,\sigma(s))$ is then future-pointing , timelike  and connects the points $(t_0, x_0)$ and $(t_0+r, x)$.

For the last statement observe that if $p=(t,x)\in I^+(p_0)$ then $x\in B^+(x_0, t-t_0)$. Let $\varepsilon=(t-t_0-\tilde d(x_0, x))/2)$.
Then one can easily check, using \eqref{ipiu}, that the open set $(t-\varepsilon,+\infty)\times  B^+(x, \varepsilon)$ is contained in $I^+(p_0)$. 
\end{proof}
As claimed by Barrett O'Neill \cite[p. 293]{O'Neil}, a fundamental problem in a Lorentzian manifold is to determine which pairs of points can be joined by a timelike curve. Our aim, next, is to  extend to standard static Finsler spacetimes \cite[Prop. 10.46]{O'Neil} 
 stating that 
there are timelike curves from $p$ to $q$ arbitrarily near to every  causal curve $\gamma$ which is not a lightlike pregeodesic.  This properties has been already proved in \cite[Prop. 7.6]{amir}, anyway we give here an elementary  proof that exploits the splitting $\R\times M$. 

Let $\gamma\colon[a,b]\to \tilde M$, $\gamma(s)=(\theta(s),\sigma(s))$ be a curve  and $\psi:(-\varepsilon,\varepsilon)\times[a,b]\rightarrow\tilde M$, $\psi(w,s)=(\zeta(w,s),\eta(w,s))$, be a variation of $\gamma$ with variational vector field $W=(Y,Z)$. Let us denote by $\psi_w$, $w\in (-\varepsilon,\varepsilon)$,  the (longitudinal) curve in the variation given by $\psi_w\colon [a,b]\to \tilde M$, $\psi_w(s)=\psi(w,s)$ and by $\dot \psi_w$ its velocity, $\dot \psi_w=\frac{\partial}{\partial s}\psi(w,s)$.  
\begin{lem}\label{negativevariation}
Let $\gamma$ be a continuous piecewise smooth causal curve  and $\psi$  a variation of $\gamma$ such that $\frac{\partial}{\partial w}L(\dot\psi_{w})|_{w=0} <0$ 
then, for sufficiently small $w>0$, the associated longitudinal curve $\psi_w$   is timelike. 
\end{lem}
\begin{proof}
As  
\[\frac{\partial}{\partial w}\tilde{g}(\dot\psi_{w},\dot\psi_{w})|_{w=0}=\frac{\partial}{\partial w}L(\dot\psi_{w})|_{w=0} <0.\] 
and $\tilde{g}(\dot\gamma,\dot\gamma)\leq 0$ the thesis immediately follows. 
\end{proof}
In any chart  $\big((\R\times U)\times (\R\times\R^n), (t,x^1,\ldots, x^n, \tau, v^1,\ldots,v^n)\big)$ of $T\tilde M$, such that $\gamma([a,b])\cap (\R\times U)\neq \emptyset$, we have 
\begin{align*}
\frac{\partial}{\partial w}L(\dot\psi_{w})&=\frac{\partial}{\partial w}\left(-\Lambda(\eta_w)\dot\zeta_{w}^2+F^2(\dot\eta_w)\right)\\&=\left(-\frac{\partial\Lambda}{\partial x^i}(\eta_w)\dot\zeta_{w}^2+\frac{\partial F^2}{\partial x^i}(\dot\eta_w)\right)\frac{\partial \eta_{w}^i}{\partial w}-2\Lambda(\eta_w)\dot\zeta_w\frac{\partial\dot\zeta_w}{\partial w}+\frac{\partial F^2}{\partial v^i}(\dot\eta_w)\frac{\partial\dot\eta^{i}_w}{\partial w},
\end{align*}
up to the finite number of instants $s_j$ where, eventually, $\dot\gamma$ has breaks (clearly, the above equation is satisfied, separately, in some intervals of the type $(s_j-\varepsilon_{1j}, s_j]$, $[s_j, s_{j}+\varepsilon_{2j})$, $\varepsilon_{1j}, \varepsilon_{2j}>0$ such that $\gamma\big((s_j-\varepsilon_{1j}, s_{j}+\varepsilon_{2j})\big)\subset \R\times U$).

For $w=0$, the above equation yields 
\begin{equation}\label{derivata}
\frac{\partial}{\partial w}L(\dot\psi_{w})|_{w=0}=\left(-\frac{\partial\Lambda}{\partial x^i}(\sigma)\dot\theta^2+\frac{\partial F^2}{\partial x^i}(\dot\sigma)\right)Z^i-2\Lambda(\sigma)\dot\theta \dot Y+\frac{\partial F^2}{\partial v^i}(\dot\sigma)\dot Z^i.
\end{equation}
Now we consider the following equation in each coordinate system as above:
\begin{equation}\label{alphaeq}
\frac{\partial F^2}{\partial v^i}(\dot\sigma)\dot Z^i+\left(\frac{\partial F^2}{\partial x^i}(\dot\sigma)-\frac{\partial\Lambda}{\partial x^i}(\sigma)\dot\theta^2\right)Z^i-2\Lambda(\sigma)\dot\theta\dot Y=-\alpha 
\end{equation}
where $\alpha$ is a positive constant.
\begin{prop}\label{timelikedeformation}
Let  $\rho\colon [a,b]\to \tilde M$, $\rho(s)=(\rho_1(s), \rho_2(s))$,  be a future-pointing causal curve of $(\tilde M,L)$ that is not a lightlike pregeodesic, then there exists  a  proper variation of $\rho$ by timelike curves.
\end{prop}
\begin{proof}
If $\rho$ is a timelike curve then the thesis follows by a simple continuity argument. Then assume that $\rho$ is not timelike. Being  $\rho$ causal and future-pointing, $\Lambda(\rho_2)\dot\rho_1>0$ on $[a,b]$, thus  we can  reparametrize $\rho$ on the same interval $[a,b]$ to obtain a curve $\gamma=\gamma(s)=(\theta(s),\sigma(s))$ such that  $\Lambda(\sigma) \dot\theta=C$ for some positive constant $C$. 
Let us consider a covering $\{\R\times U_i\}_{i\in\{1,\ldots,k\}}$ of $\gamma$ by $k$ charts of $\tilde M$ 
and  a subdivision $a=s_0<s_1<\ldots<s_k$ of the interval $[a,b]$ such that $\gamma([s_{j-1},s_j])\subset \R\times U_j$ for all $j\in \{1,\ldots,k\}$.
As $\rho$ was not a lightlike pregeodesic, necessarily  it must exist a piecewise smooth vector field $Z$ along $\sigma$, with $Z(a)=Z(b)=0$ such that 
\[\sum_{j=1}^k\int_{s_{j-1}}^{s_j}\left(\frac{\partial F^2}{\partial v^i}(\dot\sigma)\dot Z^i+\left(\frac{\partial F^2}{\partial x^i}(\dot\sigma)-\frac{\partial\Lambda}{\partial x^i}(\sigma)\dot\theta^2\right)Z^i\right)ds\neq 0,\]
otherwise, from \eqref{Meq}, $\gamma$ would be a geodesic of $(\tilde M, L)$ and then (being causal and not timelike) necessarily a lightlike one.
Clearly, up to consider the opposite vector $-Z$, we can assume that the above summation is negative.
Let us define on the  coordinate system associated to $\R\times U_j$
\[h_j= \frac{\partial F^2}{\partial v^i}(\dot\sigma)\dot Z^i+\left(\frac{\partial F^2}{\partial x^i}(\dot\sigma)-\frac{\partial\Lambda}{\partial x^i}(\sigma)\dot\theta^2\right)Z^i,\quad\quad j\in\{1,\ldots,k\},\]
then 
\[Y(s)=\frac{1}{2C}\left(\sum_{j=1}^{m(s)}\int_{s_{j-1}}^{s_j} h_j(\mu)d \mu+\int_{s_{m(s)}}^s h_{m(s)+1}d \mu+\alpha(s-a)\right),\]
where $m(s)\in \{0,\ldots,k-1\}$, such that $s\in (s_{m(s)}, s_{m(s)+1}]$ (with the convention that if $m(s)=0$ then the first term in the right-hand side is equal to $0$), solves \eqref{alphaeq} and for $\alpha=-\frac{1}{b-a}
\sum_{j=1}^{k}\int_{s_{j-1}}^{s_j} h_j(\mu)d \mu$ we get that $\alpha>0$ and $Y(b)=0$.
From \eqref{derivata}, this implies that
\begin{equation*}
\frac{\partial}{\partial w}L(\dot\psi_{w})|_{w=0}=-\alpha<0,
\end{equation*} 
and then, by Lemma~\ref{negativevariation}, we conclude. 
\end{proof}
As a consequence of Proposition~\ref{timelikedeformation} we immediately get the following fundamental properties of the relations $\ll$ and $\leq$ (compare also with \cite[Corollary1]{Minguzzi2}):
\begin{cor}\label{fundamental}
Let $p,q,z\in \tilde M$. If $p\leq q$ and $q\ll z$ (or vice versa) then $p\ll z$.
\end{cor}
Moreover, the equality between the closures of the chronological and the causal future of a point also easily follows:
\begin{cor}\label{equalclosure}
For all $p\in \tilde M$,  $\overline{J^+(p)}=\overline{I^+(p)}$ (and $\overline{J^-(p)}=\overline{I^-(p)}$).
\end{cor}
\begin{proof}
From \eqref{inclusion}, it is enough to show that $J^+(p)\subset \overline{I^+(p)}$. Let $q=(t,x)\in J^+(p)$. Clearly, if $q=p$ then $q\in  \overline{I^+(p)}$. So, let us assume that $p<q$; let $\gamma$ be a causal future-pointing curve between $p$ and $q$ and and consider a sequence of points $q_k=(t_k, x)$, with  $t_k\searrow t$. As the line $l_x$ is timelike, from Corollary~\ref{fundamental}, we get $p\ll q_k$, hence $q\in \overline{I^+(p)}$.
\end{proof}
Let us  now give a first definition in the causal hierarchy of Finsler spacetimes which, as the following ones in Definitions~\ref{cauchyhyp} and \ref{gh} are formally  the same as in Lorentzian spacetimes.
\begin{dfn}
A Finsler spacetime $(\tilde M, \tilde g)$ is {\em causally  simple} if for all $p\in\tilde M$,
$J^{\pm}(p)$ are closed and no closed future-pointing causal curve exists.
\end{dfn}
The following result extends to standard static Finsler spacetimes \cite[Th. 4.3-(a)]{CJS}.  
\begin{thm}\label{simply}
A standard static Finsler spacetime $(\tilde M, L)$  is causally simple if and only if for any couple  $(x,y)\in M\times M$ there exists a geodesic of the metric $\tilde F= F/\sqrt{\Lambda}$ joining $x$ to $y$ with length  equal to the their $\tilde F$-distance, $\tilde d(x,y)$.
\end{thm}
\begin{proof}
($\Rightarrow$)  Let $x, y\in M$ and let $\sigma\colon [0,1]\to M$ be a piecewise smooth curve curve, say, from $x$ to $y$. The curve $\gamma(s)=\left(\int_0^s \tilde F(\dot \sigma)d \mu, \sigma(s)\right)$ is then lightlike and future-pointing. This shows that $J^+(0,x)\cap l_y\neq \emptyset$. Let $t_y:=\inf\{t\in\R: (t,y)\in J^+(0,x)\}$; clearly,  $(t_y,y)\in \overline{J^+(0,x)}=J^+(0,x)$. Moreover $(t_y, y)\not\in I^+(0,x)$ otherwise, being $I^+(0,x)$ open there would exist $(t,y)\in I^+(0,x)$ with $t<  t_y$. Thus, from Proposition~\ref{timelikedeformation}, any causal future-pointing curve connecting $(0,x)$ to $(t_y, y)$ must be a lightlike geodesic  and, from Proposition~\ref{fermat}, its component on $M$ is a pregeodesic of 
$(M,\tilde F)$. Moreover, the  $(\tilde F)$-length of this component must be equal to $\tilde d(x,y)$ otherwise it would be possible to consider, as above, a lightlike future-pointing curve whose future endpoint would have $t$-coordinate less than $t_y$. 

\noindent ($\Leftarrow$) Let $(t_0,x_0)\in \tilde M$ and  $(t,x)\in \overline{J^+(t_0, x_0)}$. Consider a geodesic $\sigma:[0,1]\to M$ connecting $x_0$ to $x$ and such that $\ell(\sigma)=\tilde d(x_0,x)$. Then, the curve $\gamma(s)=\left(t_0+\int_0^s \tilde F(\dot \sigma)d \mu, \sigma(s)\right)$ is lightlike and future-pointing. Let $\{(t_k, x_k)\}$ be a  sequence of points  such that $(t_k, x_k)\in J^+(t_0,x_0)$ and $(t_k, x_k)\to (t,x)$, as  $k\to +\infty$. Consider a  sequence of future-pointing causal curves $\gamma_k$, each one connecting $(t_0,x_0)$ to $(t_k,x_k)$, so that   $\tilde d(x_0,x_k)\leq t_k-t_0$. Then, by the continuity of $\tilde d$, we get  $\tilde d(x_0,x)\leq t-t_0$. Thus, if $\tilde d(x_0,x)<t-t_0$  then, by \eqref{ipiu}, $(t_0,x_0)\ll (t,x)$ and, if $\tilde d(x_0,x)= t-t_0$ the same lightlike curve $\gamma$ gives  $(t_0,x_0)\leq (t,x)$.
\end{proof}
Since the Finslerian distance $\tilde d$ of $\tilde F$ is not symmetric, the lack of symmetry gives rise to two notions of completeness: the \emph{forward completeness} and the \emph{backward} one  (see  \cite[\S 6.2 D]{Chern}). A forward (resp. backward) Cauchy sequence in a Finsler manifold $(M,F)$ is a sequence $\{x_k\}_{k\in \N}$ with the property that  for all $\varepsilon>0$ there exists $\nu\in \N$ such that for all $k>j\geq \nu$, $d(x_j,x_k)<\varepsilon$ (resp. $d(x_k,x_j)<\varepsilon$), where $d$ is the distance associated to $F$. Then $(M,F)$ is said {\em forward} (resp. {\em backward}) {\em complete} if any forward (resp. backward) Cauchy sequence is convergent.

Forward and backward completeness of the metric $\tilde F$ are related to the existence of some particular Cauchy hypersurfaces in $(\tilde M, L)$.
\begin{dfn}\label{cauchyhyp}
A \emph{Cauchy hypersurface} in a Finsler spacetime $(\tilde M,\tilde g)$ is a topological  hypersurface that is met exactly once by every inextensible future-pointing causal curve.
\end{dfn}
The following theorem extends to standard static Finsler spacetime a well known result, already cited at the beginning of Section~\ref{2},  valid for  standard static Lorentzian spacetimes, and  \cite[Theorem 4.4]{CJS} for  standard stationary ones.
\begin{thm}\label{cauchy}
A   slice (and then any slice) $S_t=\{t\}\times M$ is a Cauchy hypersurface of the standard static Finsler spacetime $(\tilde M,  L)$ if and only the Finsler manifold $(M, \tilde F)$ is forward and backward complete.
\end{thm}
Before proving Theorem~\ref{cauchy} we need the following well known result that we report for the reader convenience.
\begin{lem}\label{length}
Let $(M,F)$ be a Finsler manifold. Then $(M,F)$ is forward (resp. backward) complete if and only if for any piecewise smooth curve of finite Finslerian length $\sigma:[a,b)\rightarrow M$ (resp. $\sigma:(a,b]\to M$) there  exists a point $x_0\in M$ such that $\sigma(s)\rightarrow x_0$, as $s\rightarrow b$ 
(resp. as $s\to a$).
\end{lem}
\begin{proof}
Let us show the equivalence only for forward completeness since the backward case is completely analogous.

\noindent $(\Rightarrow)$  Let $\ell(\sigma)\in (0,+\infty)$ be the length of $\sigma$. Set $K=\{x\in M : d(\sigma (a), x)\leq \ell(\sigma)\}$. Since $K$ is forward bounded and closed, we have that it is compact by Hopf-Rinow theorem (see \cite[Th.6.6.1]{Chern}). Fix a sequence $\{s_k\}$ in $[a,b)$ such that $s_k\rightarrow b$. Since $\sigma( [a,b))\subseteq K$, the compactness implies that there exists a point $x_0\in K$ such that $\sigma (s_{k})$ converges to $x_0$, up to subsequences. If $\lim_{s\rightarrow b}\sigma(s) \neq x_0$, there would exist an $\varepsilon > 0$ such that $\sigma$ leaves the set $\{x\in M : d(x,x_0) \leq\varepsilon\}$ infinite times while, definitively,  $d(\sigma(s_k),x_0)<\varepsilon/2$. This is a contradiction because $\sigma$ has a finite length.

\noindent$(\Leftarrow)$ We suppose that $(M,F)$ is not forward complete manifold. Therefore, there exists a geodesic $\sigma:[a,b)\rightarrow M$, hence with finite length, which is not extendible to $s=b$. 
\end{proof}
\begin{proof}[Proof of Theorem~\ref{cauchy}]
$(\Rightarrow)$ By contradiction, assume that $(M,\tilde F)$ is  not, say,  forward complete. Then, from Lemma~\ref{length}, there exists $\sigma: [a,b) \rightarrow M$ having finite length, $\tilde \ell(\sigma)=\int_a^b\tilde F(\dot\sigma)ds$,  and such   that it  does not converge as $s\rightarrow b$.  Take $t_0\in \R$ such that $t_0+\ell(\sigma)<t$ and define $\theta: [a,b)\rightarrow \R$, $\theta(s)=t_0+\tilde\ell({\sigma}|_{[a,s]})$ and $\gamma:[a,b)\rightarrow \mathbb{R} \times M$, $\gamma(s)=(\theta(s),\sigma(s))$. Observe that $\gamma$ is a future-pointing lightlike curve, by definition, and it is future-inextensible because $\sigma$ is inextensible. Hence, any extension of $\gamma$ as a past inextensible causal curve does not intersect $S_t$, a contradiction.

\noindent$(\Leftarrow)$ Let $(t_{0},x_{0})\in\R \times M$, with $t_{0}< t$. We consider a future-inextensible causal curve $\gamma:[a,b)\rightarrow \R \times M $, $\gamma(s)=(\theta(s),\sigma(s))$ that starts from $(t_{0},x_{0})$. As $\theta$ is an increasing function, $\gamma$ cannot intersect $S_t$ more than once. So assume, by contradiction,   that $\gamma$ does not meet $S_t$  then,  as $\theta$ is also continuous, $\lim_{s\to b}\theta(s)\in  (t_0,t]$. Moreover, being $\gamma$ causal, $\int_a^s \tilde F(\dot\sigma)d \mu \leq \theta(s)-t_0$, for all $s\in [a,b)$. Hence, there exists $\lim_{s\to b^-}\int_a^s \tilde F(\dot\sigma)d \mu\in[0,t-t_0]$, i.e $\sigma$  has a finite length and therefore from Lemma~\ref{length} it is extendible. Hence,  $\gamma$ is future extendible. The proof for a past-inextensible curve is analogous.
\end{proof}
Notice that, as $L$ and $\tilde L$ share the same causal curves, $S_t=\{t\}\times M$ is a Cauchy hypersurface also for $(\tilde M,\tilde L)$. Anyway,  the completeness assumptions in Theorem~\ref{cauchy} involve the optical metric $\tilde F$ and not $F$. We can give  some conditions on $\Lambda$ ensuring  forward and  backward completeness of the conformal metric $\tilde F=F/\sqrt{\Lambda}$, provided that $F$ is forward and backward complete. To this end, let us introduce the following notions. 

Given a Finsler manifold $(M,F)$ and a  function $\Lambda:M\to \R$, we say that   $\Lambda$ is \emph{forward} (resp. {\em backward}) {\em subquadratic}, if there exist constants $c_1,c_2>0$ such that, for all $x\in M$,  $|\Lambda(x)|\leq c_1d(\bar x,x)^2+c_2$ ($|\Lambda(x)|\leq c_1 d(x, \bar x)^2+c_2)$ for some $\bar x\in M$. 
Clearly, by the triangle inequality, one immediately sees that  previous definitions are independent  from the point $\bar x$ (up to change the constants $c_1$ and $c_2$). Moreover, we say that   $\Lambda$ is {\em subquadratic} if it is both  forward and backward subquadratic.
\begin{prop}
Let $(M,F)$ be a forward (resp. backward) complete Finsler manifold. If $\Lambda$ is a positive smooth forward (resp. backward) subquadratic function then $(M,\tilde F)$ is forward (resp. backward) complete.
\end{prop}
\begin{proof}
Let $\{x_k\}_{k\in\N}$ be a forward Cauchy sequence for $(M,\tilde F)$ and let us denote by $\tilde d$ the distance induced by $\tilde F$. We can find two indexes $\nu$ and $\bar m$ such that, for all $k>\bar m>\nu$, $\tilde d(x_{\bar m},x_{k})<\frac{1}{2\sqrt{c_1}}$. Let $\sigma_{k}:[0,1]\rightarrow M$ be a curve  joining $x_{\bar m}$ to $x_{k}$ and such that
\begin{equation*}
\int_{0}^{1}\!\!\tilde F(\dot\sigma_k)ds <\frac{1}{2\sqrt{c_1}}.
\end{equation*}
Now we evaluate the $F$-distance $d(x_{\bar m}, x_k)$:
\begin{align*}
\int_{0}^{1}\!\!F(\dot\sigma_k)ds&= \int_{0}^{1}\!\!\frac{1}{\sqrt{\Lambda(\sigma_k)}}F(\dot\sigma_k)\sqrt{\Lambda(\sigma_k)}\, d s\\
&\leq \int_{0}^{1}\!\!\tilde F(\dot\sigma_k)(c_1 d(\bar x,\sigma_k)^2+c_2)^{\frac 1 2} ds \\
&\leq\int_{0}^{1}\tilde F(\dot\sigma_k)(c_1(2d(x_{\bar m}, \sigma_k)^2+2d(\bar x,x_{\bar m})^2)+c_2)^{\frac1 2 }\, d s\\
&\leq\int_{0}^{1}\tilde F(\dot\sigma_k)\left (2c_1\left(\int_{0}^{s}\ F(\dot\sigma_k)d \mu\right)^2+2c_1d(\bar x,x_{\bar m})^2+c_2\right)^{\frac 12}\!\!ds \\
&<\frac{1}{\sqrt{2}}\int_{0}^{1}\ F(\dot\sigma_k)\, ds\left (1+\frac{2c_1d(\bar x, x_{\bar m})^2+c_2}{2c_1\int_{0}^{1}\ F(\dot\sigma_k)d s}\right)^{\frac1 2}\!\!ds.
\end{align*}
This  shows that the sequence
$\{\int_{0}^{1}\ F(\dot\sigma_k)ds \}_{m\in\mathbb{N}}$
is bounded. Thus $\{x_k\}$ is forward bounded with respect to $d$ and therefore it admits a converging subsequence. The proof in the backward case is analogous.
\end{proof}
From the above proposition and Theorem~\ref{cauchy} we get:
\begin{cor}
If $(M, F)$ is forward and backward complete and $\Lambda$ is subquadratic, then $S_t=\{t\}\times M$ is a Cauchy hypersurface in $(\tilde M,L)$.
\end{cor}
The strongest property in the causal hierarchy of spacetimes is global hyperbolicity:
\begin{dfn}\label{gh}
A Finsler spacetime $(\tilde{M},\tilde g)$ is \emph{globally hyperbolic} if it is causal (i.e. no future-pointing closed causal curve exists) and $J^{+}(p)\cap J^{-}(q)$ is a compact subset, for every $p, q\in \tilde M$.
\end{dfn}
Classically, a spacetime is globally hyperbolic if it is  strongly causal (namely, no almost closed future-pointing causal curve exists).  As shown recently in \cite{BerSan08},  the weaker property of being causal  can be used instead. Anyway, in a standard static Finsler spacetime both conditions  hold because every future-pointing causal curve  has increasing $t$-component.  The following theorem is the analogous of \cite[Th. 4.3-(b)]{CJS} which concerns   standard stationary Lorentzian spacetimes.
\begin{thm}\label{ghth}
A standard static Finsler spacetime  $(\tilde{M},L)$ is globally hyperbolic if and only if $ \bar B^{+}(x,r)\cap \bar B^{-}(y,s)$ is compact, for every $x,y\in M$ and $r,s>0$, where $\bar B^{\pm}$ are the the closure of the forward and backward balls on $M$ w.r.t. to the distance $\tilde d$ associated to $F/\sqrt{\Lambda}$.
\end{thm}
%
\begin{proof}
($\Rightarrow$) Let us first show that $J^{+}(t_0,x_0)$ is closed  for all $(t_0,x_0)\in \tilde{M}$. Let $(t,x)\in \overline{J^{+}(t_0,x_0)}$. Clearly we can find $(t_1,x_1)$ such that $(t,x)\in \overline {J^{+}(t_0,x_0)}\cap I^{-}(t_1,x_1)$. Then, being $I^-(t_1,x_1)$ open, any sequence $\{(t_k, x_k)\}\subset J^+(t_0,x_0)$ and converging to $(t,x)$ is definitively contained in $I^{-}(t_1,x_1)$ and admits, by global hyperbolicity, a converging subsequence to a point $(\tilde t,\tilde x)\in J^{+}(t_0,x_0)\cap J^{-}(t_1,x_1)$. As $(\tilde t, \tilde x)$ must be equal to $(t,x)$, it follows that $J^+(t_0,x_0)$ is closed.  Analogously, the same holds for $J^{-}(t_0,x_0)$. Thus, for any $(r,x)\in \tilde M$, by Proposition \ref{forwball} and Corollary~\ref{equalclosure}, we get
\begin{multline*}
\{0\}\times\left(\bar{B}^{+}(x,r)\cap\bar{B}^{-}(x,r)\right)= \\\overline{I^{+}(-r,x)}\cap \overline{I^{-}(r,x)}= \overline {J^{+}(-r,x)}\cap \overline{J^{-}(r,x)}=J^{+}(-r,x)\cap J^{-}(r,x).
\end{multline*} 
Hence, the left-hand side is a compact set and this can be easily seen to be equivalent to the thesis (see, e.g. \cite[Prop. 2.2]{CJS}).

\noindent ($\Leftarrow$) For every $(t_0,x_0),(t_1,x_1)\in \tilde{M}$,
by Proposition \ref{forwball} and Corollary~\ref{equalclosure}, we get: 
\[
\overline{J^{+}(t_0,x_0)}\cap \overline{J^{-}(t_1,x_1)}=\bigcup_{s\in[0,t_1-t_0]}\{t_0+s\}\times \left(\bar{B}^{+}(x_0,s)\cap\bar{B}^{-}(x_1,t_1-t_0-s)\right).
\]
As in \cite[p.~936]{CJS}, one obtains that the right-hand side is a compact subset and, then, $\overline{J^{+}(t_0,x_0)}\cap \overline{J^{-}(t_1,x_1)}$ is compact. From \cite[Prop. 2.2 and  Th. 5.2]{CJS}, we know that the compactness of $\bar B^+(x,r)\cap \bar B^-(y,s)$, for any $x, y\in M$ and $r,s>0$,  implies that any pair of points $x,y\in M$ can be joined by a geodesic from $x$ to $y$ with length $\tilde d(x,y)$.  Hence,  from  Theorem~\ref{simply}, both $J^{+}(t_0,x_0)$ and $J^{-}(t_1,x_1)$ are closed and this concludes the proof. 
\end{proof}
\begin{rmk}
Notice that, in the proof of Theorem~\ref{ghth} we have also shown that global hyperbolicity implies causal simplicity as in the causal ladder for Lorentzian spacetime. From the Finslerian Hopf-Rinow theorem (see \cite[Th.6.6.1]{Chern}), it is clear that   forward and backward completeness of the metric $\tilde F$ implies  compactness of the intersections of the closed balls (in other words,  the fact that a slices $S_t$ is a  Cauchy hypersurface implies that 
$(\tilde M, L)$ is globally hyperbolic). Anyway,  the latter condition is weaker than forward or backward completeness (see \cite[Example~4.6]{CJS}), so a standard static Finsler spacetime might be globally hyperbolic but with slices $S_t$ which are not Cauchy hypersurface. It is well known that in a Lorentzian spacetime $\tilde M$,  global hyperbolicity is equivalent to the existence of a {\em smooth},  {\em spacelike} Cauchy hypersurface \cite{BerSan03} and a {\em Cauchy temporal function} \cite{BerSan05}, i.e. a  {\em smooth} function $f\colon \tilde M\to\R$ which is strictly increasing on any future-pointing causal curve and whose  level sets are  (spacelike, smooth) Cauchy hypersurfaces. As observed  in \cite[Theorem 1]{javsan14}, the existence of a Cauchy temporal function  follows also in a  Finsler spacetime by using weak KAM theory as shown in \cite{FatSic12}. Indeed,  \cite[Th. 1.3]{FatSic12} extends the result in \cite{BerSan05}  to a manifold $\tilde M$ endowed with a cone structure $\mathcal C$, i.e. a  continuous map (w.r.t. the Hausdorff metric in local coordinates) $p\in\tilde M \mapsto \mathcal C_p\subset T_x \tilde M$, where $\mathcal C_p$  is a closed convex cone with vertex at $0$, with non-empty interior and not containing any  complete affine line.  Clearly, from Lemma~\ref{convexcausal}, the  set of future-pointing causal vectors plus the zero section in a standard static Finsler spacetime  defines a cone structure.
It is however worth pointing out  that, when applied to Finsler spacetimes, \cite[Th. 1.3]{FatSic12} does not automatically gives a spacelike hypersurface (in the sense of Definition~\ref{char}). This is due to the fact that a cone structure takes into account only future-pointing directions and the  cones might flatten ``at infinity'',   i.e. along some lines corresponding to future-pointing lightlike vectors  at the points outside an arbitrarily large  compact subset of  $\tilde M$. Thus, in principle, the tangent bundle of a level set of a temporal function might contain vectors which are causal and past-pointing. 
\end{rmk}
For this reason, we introduce the following:
\begin{dfn}
Let $(\tilde M,L)$ be a Finsler spacetime. A  hypersurface $\mathcal H\subset \tilde M$ is said {\em future spacelike} if $T\mathcal H$ does not contain any future-pointing causal vector.
\end{dfn}
We can now show that a future Cauchy hypersurface can be always constructed as the graph of a smooth function on $M$ (as for a standard stationary Lorentzian spacetime, see the proof of \cite[Th. 5.10]{CJS}). 
\begin{prop}\label{ghcauchy}
Let $(\tilde M, L)$ be a globally hyperbolic standard static Finsler spacetime such that $\tilde F$ is not forward or backward complete. Then there exists a smooth function $f: M\to \R$ such that $S_f=\{(f(x), x):x \in M\}$ is a future spacelike, smooth, Cauchy hypersurface.
\end{prop}
\begin{proof}
As $(\tilde M, L)$ is globally hyperbolic, from Theorem~\ref{ghth}, the sets $\bar B^+(x,r)\cap \bar B^-(x, r)$ are compact for any $x\in M$ and $r>0$. Then from \cite[Th. 1]{Matvee13},   there exists a smooth function $f\colon M\to \R$ such that $\tilde F-df $ is a backward and forward complete Finsler metric on $M$ (in particular, $\tilde F(v)-d f(v)>0$ for all $v\in TM\setminus 0$).   
Notice that for all $\tilde v\in TS_f$, $\tilde v=(df(v),v)$, $v\in TM$, we have $L(\tilde v)=-\Lambda(x)\big(df(v)\big)^2+F^2(v)$, where $x=\tilde \pi(\tilde v)$. Being  $\tilde F(v)-d f(v)>0$ for all $v\in TM\setminus 0$, we conclude that $L((df(v),v))>0$ for all $v$ such that $df(v)>0$ or, equivalently, that $S_f$ is a future spacelike hypersurface. 
Let $\gamma\colon(a,b)\to \tilde M$, $\gamma(s)=(\theta(s), \sigma(s))$, be an inextensible causal curve and assume that $\gamma$ intersects $S_f$ at least twice at the instants $s_1,s_2\in (a,b)$, $s_1<s_2$. Thus, we have
\[0=\int_{s_1}^{s_2}\big(\dot\theta -df(\dot\sigma)\big)d s\geq \int_{s_1}^{s_2} \big(\tilde F(\dot\sigma)-d f(\dot\sigma)\big)ds,\]
 which is possible if and only if $\sigma$ is constant on $[s_1,s_2]$ but this is a contradiction because $S_f$ is the graph of $f$. Let now assume that $\gamma$ does not intersect $S_f$. Then, being $\theta$ and $\sigma$ continuous, it must be either  $\theta-f\circ \sigma<0$ or $\theta -f\circ\sigma >0$ on $(a,b)$. Assume that the former inequality holds and take $s_0\in (a,b)$.
We have, for any $s>s_0$,
\[\int_{s_0}^s \tilde F(\dot\sigma)-df(\dot\sigma)d\mu\leq \theta(s)-\theta(s_0)-f(\sigma(s))+f(\sigma(s_0))< f(\sigma(s_0))-\theta(s_0).\]
Therefore $\sigma|_{[s_0,b)}$ has finite length w.r.t.  $\tilde F-df$. As this metric is forward complete, from Lemma~\ref{length}, $\sigma$ is extendible in $b$, i.e.
there exist $x_b\in M$ such that $x_b=\lim_{s\to b}\sigma(s)$.
So, by continuity of $f$ and  monotonicity of $\theta$ also $\theta$ is extendible in $b$, which contradicts the fact that $\gamma$ was future inextensible. By a  similar reasoning,  using that $\tilde F-df$ is also backward complete and $\gamma$ is past inextensible, we obtain  that also the second inequality cannot hold. 
\end{proof}
The following proposition shows that, under the  condition of finite {\em reversibility} \eqref{reversibility} (reversibility  is a measure of how much a Finsler metric is far from being reversible;  it was introduced in \cite{Radema04}), it is possible to modify $f$ to get a smooth, spacelike (not only future spacelike), Cauchy hypersurface. 
\begin{prop}\label{spacelikeCauchy}
Under the assumptions and notations of Proposition~\ref{ghcauchy}, assume also that
\begin{equation}\label{reversibility}
\alpha:=\sup_{v\in TM\setminus 0} \frac{F(v)}{F(-v)}<+\infty,
\end{equation}
then $S_{f/\alpha}$ is a spacelike Cauchy hypersurface.
\end{prop}
\begin{proof}
We show that, under \eqref{reversibility},   both $\tilde F-df/\alpha$ and $\tilde F^- -df/\alpha$,  are Finsler metric with the former which is also forward and backward complete, where $\tilde F^-$ denotes the reverse Finsler metric associated to $\tilde F$, $\tilde F^-(v):=\tilde F(-v)$, and $f$ is a function such that 
$\tilde F-df$ is a forward and backward complete Finsler metric,  \cite[Th. 1]{Matvee13}. Observe that $\alpha\geq 1$ and it is equal to $1$ if and only if $F$ is a reversible Finsler metric, i.e. $F(v)=F(-v)$ for all $v\in TM$.   Clearly, we have $\tilde F(v)-\frac{df(v)}{\alpha}>0$, for all $v\in TM\setminus 0$. Moreover, for all $v\in TM\setminus 0$ with $df(v)\geq0$, we have $\frac{df(v)}{\alpha}\leq \frac{df(v)F(-v)}{F(v)}=\frac{df(v)\tilde F(-v)}{\tilde F(v)}<\tilde F(-v)=\tilde F^-(v)$. On the other hand,  if $df(v)<0$ then $\tilde F^-(v)-\frac{df(v )}{\alpha}>0$. Hence, both $\tilde F-df/\alpha$ and $\tilde F^- -df/\alpha$, are Finsler metric (see, e.g., \cite[Cor. 4.17]{JavSan14}). Clearly, $\tilde F-df/\alpha$ remains backward and forward complete because it is  $1/\alpha$-homothetic, to $\alpha \tilde F-df$. Thus the same proof of Proposition~\ref{ghcauchy} shows that $S_{f/\alpha}$ is a future spacelike Cauchy hypersurface. Actually, for $df(v)<0$, we also have 
$\tilde F^-(-v)>-\frac{df(v)}{\alpha}$, hence $-\tilde F(v)<\frac{df(v)}{\alpha}$ which is equivalent, together with  $\tilde F(v)>\frac{df(v)}{\alpha}$, to $F^2(v)-\Lambda(x)\frac{(df(v))^2}{\alpha^2}>0$, for all $v\in T_x M$ and $x\in M$. Since,  $TS_{f/\alpha}=\{(\frac{df(v)}{\alpha}, v):v\in TM\}$ and $L((\frac{df(v)}{\alpha}, v))=F^2(v)-\Lambda(x)\frac{(df(v))^2}{\alpha^2}$, we have  done.
\end{proof}
We have seen in Theorem~\ref{simply} that causal simplicity is equivalent to the existence of a $\tilde F$-length minimizing geodesic between any couple of points $(x,y)\in M\times M$ and then from Proposition~\ref{fermat} and Remark~\ref{arrival} to the existence of a future-pointing lightlike geodesic between, $(t_0,x)$ and $l_y$ (for any $t_0\in\R$) which minimizes the arrival time functional $T_{(t_0,x)l_y}$.  Under non-trivial topology of the manifold $M$ we can obtain also a multiplicity result. 
\begin{cor}
Let $(\tilde M, L)$ be a globally hyperbolic standard static Finsler spacetime such that $M$ is not contractible. Then for any $p, q\in \tilde M$ there exist infinitely many lightlike future-pointing geodesics $\{\gamma_k\}_{k\in\N}$ connecting $p$ to $l_{x_q}$ and such that 
$T_{pl_{x_q}}(\gamma_k)\to +\infty$, as $k\to +\infty$. 
\end{cor}
\begin{proof}
From global hyperbolicity, $\bar B^+(x,r)\cap \bar B^-(y,s)$ is compact for any $x,y\in M$ and $r,s>0$. Then from \cite[Th. 5.2]{CJS} we know that there exists a sequence $\{\sigma_k\}$ of geodesics for $\tilde F$ from $x_p$ to $x_q$ such that $\tilde \ell(\sigma_k)\to +\infty$. Thus, it is enough to apply Proposition~\ref{fermat} and Remark~\ref{arrival}.
\end{proof}
\begin{rmk}
The above corollary is based on L\"usternik-Schnirelmann theory for Finsler geodesics as developed in \cite{CaJaMa11}. Indeed the fact that $M$ is not contractible implies that the L\"usternik-Schnirelmann category of the based loop space is equal to $+\infty$ and  compact subsets of it, with arbitrarily large category, do exist \cite{FadHus91}. 
Similar results concerning existence and multiplicity of geodesics of $(\tilde M, L)$ between two given events can be obtained by matching  the results in \cite{mas} and the ones in  \cite{CaJaMa11}.
\end{rmk}
\section{Conclusions and further discussions}
We have extended the class of standard static Lorentzian  spacetimes $\R\times M$ by considering Finslerian  optical metrics.   A spherical symmetric Finsler optical metric has been proposed in \cite{LPH} to get a perturbation of the Schwarzschild metric. As observed in \cite{LPH}, the most important feature that distinguishes a Finsler metric from a Lorentzian one is the fact that  it breaks spacetime isotropy also at an infinitesimal scale. In the case of a standard static Finsler spacetime, the anisotropy is  confined to the restspaces $\{t\}\times M$ (where $t$ is the natural coordinate on $\R$),   relative to the observer field $\partial_t$, through  the optical metric $F/\Lambda$.   
We also observe that the map   $(t,x)\mapsto (-t,\varphi(x))$ of $\tilde M$, where $\varphi$ is a diffeomorphism of $M$,  flips the future causal cone in the past one at $T_{(-t,\varphi(x))}\tilde M$. If we assume that $\Lambda$ is constant and $F$ is a locally Minkowski, non-reversible, metric on $M$ (i.e. there exists a covering of $M$ by coordinate systems such that, in the corresponding natural coordinates $(x^1,\ldots, x^n,v^1,\ldots, v^n)$ of $TM$, $F$ depends only on $(v^i)$) then the causal structure of $\tilde M$ is not invariant under the $PT$-transformation $(\tau,v^1,\ldots,v^n)\mapsto (-\tau,-v^1,\ldots,-v^n)$. So such a particular class of standard static Finsler spacetimes   could be interesting for  some  extensions of the Standard Model of particles when $CPT$ symmetry and Lorentz-invariance  violation are considered \cite{Russell2,ChaWan12}.

A wider class of splitting Finsler spacetimes  $\tilde M=\R\times M$ can be obtained, taking  a one-form $\omega$ on $M$ and a function $\Lambda:M\to \R$ which is non-necessarily positive,  by considering  the quadratic Finsler function $L:T\tilde M\to \R$,
\[L((\tau,v))=-\Lambda\tau^2+2\tau \omega(v)+F^2(v).\]
It can be easily seen (compare with \cite[Prop. 3.3]{CJS2}) that the fundamental tensor $\tilde g=-\Lambda  dt^2 + \omega\otimes dt +dt\otimes \omega+g$ 
of $L$, where $g$ is the fundamental tensor of $F$,  is a smooth, symmetric section of index $1$  of the tensor bundle  $\pi^*(T\tilde M)\otimes \pi^*(T\tilde M)$ over $T\tilde M\setminus \mathcal T$,  where $\mathcal T$ is again the line bundle defined by $\partial_t$,  
if and only if $\Lambda(x)+\|\omega\|_x>0$ for all $x\in M$, where 
\[\|\omega\|_x=\min_{v\in T_xM\setminus\{0\}} \max_{w\in T_xM\setminus\{0\}} \frac{|\omega_x(w)|}{\sqrt{g_v(w,w)}},\]
(so, in particular, if $\Lambda$ is a positive function).
The vector field $\partial_t$ is  still  a Killing vector field but  it is no more timelike (it is spacelike  at the points $(t,x)$ where $\Lambda(x)<0$ and lightlike iff $\Lambda(x)=0$) and no more $\tilde g_{(\tau, v)}$-orthogonal to the vectors in $\{0\}\times T_x M$, $x=\pi^M(v)$ (recall Remark~\ref{orto}). On the other hand, the function $t\colon\tilde M\to \R$ is a temporal function in the sense that it is smooth and $dt(\tilde v)\neq 0$ for any causal vector $\tilde v\in T\tilde M$. This allows us to define future-pointing causal vectors as the  ones  such that $dt(\tilde v)>0$. So, for a function $\Lambda$ which is not positive our definition of a Finsler spacetime (Definition~\ref{fst}) should be now  intended in the weaker sense of Remark~\ref{sign}: $\tilde g_{\tilde v}(\tilde v,\tilde v)$ is positive in a punctured neighbourhood of $\mathcal T_p$, for any $p=(t,x)\in \tilde M$ such that $\Lambda(x)>0$, while there always exist vectors $\tilde v_1, \tilde v_2$, in any punctured neighbourhood of $\mathcal T_{(t,x)}$, such that $\tilde g_{\tilde v_1}(\tilde v_1,\tilde v_1)>0$ and $\tilde g_{\tilde v_2}(\tilde v_2,\tilde v_2)<0$, iff $\Lambda(x)=0$.      
In the case where $\Lambda$ is positive, such  Finsler spacetimes  come into play: (a) already when one consider a different splitting of the type $\R\times S_f$, in standard static Finsler spacetime, after a coordinate change of  the type $(t,x)\mapsto (t+f(x), x)$ (recall Proposition~\ref{spacelikeCauchy}); (b) as a Finslerian generalization of a standard stationary Lorentzian spacetime where the Riemannian metric $g_0$ is replaced by the fundamental tensor $g$ of the Finsler metric $F$.
In the general case where $\Lambda$ is not positive, they represent the spacetime counterpart of wind Finslerian structures morally in the same way as SSTK splittings correspond to  wind Riemannian structures (see \cite{CJS2}) and  they certainly deserve further study. In particular, a challenging task is the study of the relations between their causality properties and the geometry of their optical structure (associated to  future-pointing lightlike vectors). The latter  is described, in the general case of a non-positive $\Lambda$,   by two conic pseudo-Finsler metrics  (in the sense of \cite{JavSan14}):
\begin{align}
F^o(v) &= \frac{F^2(v)}{-\omega(v)+\sqrt{\Lambda
		F^2(v)+\omega(v)^2}},\label{cf}\\
F^o_l(v) &= -\frac{F^2(v)}{\omega(v)+\sqrt{\Lambda
		F^2(v)+\omega(v)^2}},\label{lf}
\end{align}
The cone structure $\mathcal A\subset TM$, where the first metric is defined, is given by the whole tangent space $T_xM$, for all $x\in M$ with $\Lambda(x)>0$, by the vectors $v$ such that $\omega_x(v)<0$ plus $0$ at the points where $\Lambda(x)=0$ and  by the vectors $\{v\in TM_l:\omega(v)<0, \Lambda
F^2(v)+\omega(v)^2\geq 0\}$ on $M_l$. This last cone structure  is also the set  where $F^0_l$ is defined and we have, there,   $0< F^o(v)\leq F_l^o(v)$,  with equality if and only if $\Lambda
F^2(v)+\omega(v)^2= 0$. This is the condition that the projections of the lightlike vectors of $(\tilde M,L)$ on $TM_l$ must satisfy. The break of continuity  (see \cite[Def. 2.25]{FatSic12}) of the map $x\in M\mapsto \mathcal A_x$ at the critical region $\{x\in M:\Lambda(x)=0\}$  is certainly a difficulty in dealing with such cone structures but  the spacetime point of view, as shown in \cite{CJS2}, is of help since the map that  to $p\in\tilde M$ associates the set of future-pointing causal vectors at $p$ is instead continuous.

\section*{Acknowledgements}
We would like to thank the referee for his stimulating comments and for pointing out some references. We also thank some interesting  comments and remarks by M. A. Javaloyes and M. S\'anchez.

\medskip
\noindent EC is a member of and has been partially supported during this research  by the ``Gruppo Nazionale per l'Analisi Matematica, la Probabilit\`a e le loro Applicazioni'' (GNAMPA) of the ``Istituto Nazionale di Alta Matematica (INdAM)'',  by 
``FRA2011, Politecnico di Bari'', and by the project MTM2013- 47828-C2-1-P (Spanish MINECO with FEDER funds).

\medskip
\noindent GS is a member of the ``Gruppo Nazionale per le Strutture Algebriche Geometriche e loro Applicazioni'' (GNSAGA).

\end{document}